\pdfoutput=1
\RequirePackage{ifpdf}
\ifpdf 
\documentclass[pdftex]{sigma}
\else
\documentclass{sigma}
\fi

\numberwithin{equation}{section}

\newtheorem{Theorem}{Theorem}[section]
\newtheorem{Corollary}[Theorem]{Corollary}
\newtheorem{Lemma}[Theorem]{Lemma}
\newtheorem{Proposition}[Theorem]{Proposition}
 { \theoremstyle{definition}

\newtheorem{Remark}[Theorem]{Remark} }

\DeclareMathOperator{\Iso}{\mathrm{Isom}}

\newcommand{\C}{\ensuremath{\mathbb{C}}}

\renewcommand{\O}{\ensuremath{\mathrm{O}}}
\renewcommand{\SS}{\ensuremath{\mathbb{S}}}

\newcommand{\bmat}{\begin{pmatrix}}
\newcommand{\emat}{\end{pmatrix}}

\newcommand{\GL}{\mathrm{GL}}
\newcommand{\1}{\mathbf{1}}

\newcommand{\e}{\mathrm{e}}

\renewcommand{\d}{{\mathrm d}}
\newcommand{\bcase}{\begin{case}}
\newcommand{\ecase}{\end{case}}

\newcommand{\bclaim}{\begin{claim}}
\newcommand{\eclaim}{\end{claim}}

\newcommand{\bstep}{\begin{step}}
\newcommand{\estep}{\end{step}}

\newcommand{\bhlem}{\begin{hlem}}
\newcommand{\ehlem}{\end{hlem}}

\newcommand{\bleer}{\begin{leer}}
\newcommand{\eleer}{\end{leer}}

\newcommand{\ol}{\overline}

\newcommand{\bs}{\begin{Proposition}}
\newcommand{\es}{\end{Proposition}}
\newcommand{\btheo}{\begin{Theorem}}
\newcommand{\etheo}{\end{Theorem}}
\newcommand{\bfolg}{\begin{Corollary}}
\newcommand{\efolg}{\end{Corollary}}
\newcommand{\blem}{\begin{Lemma}}
\newcommand{\elem}{\end{Lemma}}
\newcommand{\bnote}{\begin{note}}
\newcommand{\enote}{\end{note}}
\newcommand{\bprf}{\begin{proof}}
\newcommand{\eprf}{\end{proof}}
\newcommand{\bd}{\begin{displaymath}}
\newcommand{\ed}{\end{displaymath}}
\newcommand{\be}{\begin{eqnarray*}}
\newcommand{\ee}{\end{eqnarray*}}
\newcommand{\eeqa}{\end{eqnarray}}
\newcommand{\beqa}{\begin{eqnarray}}
\newcommand{\bi}{\begin{itemize}}
\newcommand{\ei}{\end{itemize}}
\newcommand{\bnum}{\begin{enumerate}}
\newcommand{\enum}{\end{enumerate}}

\newcommand{\beq}{\begin{equation}}
\newcommand{\eeq}{\end{equation}}

\newcommand{\rr}{\mathbb{R}}

\newcommand{\vf}{\varphi}
\newcommand{\earr}{\end{array}\]}
\newcommand{\barr}{\[\begin{array}}
\newcommand{\bvec}{\left(\begin{array}{c}}
\newcommand{\evec}{\end{array}\right)}

\newcommand{\g}{\mathfrak{g}}
\renewcommand{\a}{\mathfrak{a}}

\newcommand{\h}{\mathfrak{h}}

\newcommand{\m}{\mathfrak{m}}

\newcommand{\hol}{\mathfrak{hol}}

\newcommand{\+}{\oplus}

\newcommand{\so}{\mathfrak{so}}

\newcommand{\del}{\partial}

\newcommand{\bbem}{\begin{Remark}}
\newcommand{\ebem}{\end{Remark}}
\newcommand{\bbez}{\begin{bez}}
\newcommand{\ebez}{\end{bez}}
\newcommand{\bbsp}{\begin{bsp}}
\newcommand{\ebsp}{\end{bsp}}

\newcommand{\tM}{\widetilde{M}}

\DeclareMathOperator{\Span}{\mathrm{span}}

\newcommand{\hook}{\makebox[7pt]{\rule{6pt}{.3pt}\rule{.3pt}{5pt}}\,}

\newcommand{\R}{\mathbb{R}}

\newcommand{\belabel}[1]{\begin{equation}\label{#1}}

\begin{document}


\renewcommand{\thefootnote}{}

\newcommand{\arXivNumber}{2309.11184}

\renewcommand{\PaperNumber}{084}

\FirstPageHeading

\ShortArticleName{Pseudo-K\"{a}hler Manifolds with Essential Conformal Transformations}

\ArticleName{Compact Locally Conformally Pseudo-K\"{a}hler\\ Manifolds with Essential Conformal Transformations\footnote{This paper is a~contribution to the Special Issue on Global Analysis on Manifolds in honor of Christian B\"ar for his 60th birthday. The~full collection is available at \href{https://www.emis.de/journals/SIGMA/Baer.html}{https://www.emis.de/journals/SIGMA/Baer.html}}}

\Author{Vicente CORT\'ES~$^{\rm a}$ and Thomas LEISTNER~$^{\rm b}$}

\AuthorNameForHeading{V.~Cort\'es and T.~Leistner}

\Address{$^{\rm a)}$~Department Mathematik, University of Hamburg, Bundesstra{\ss}e 55, 20146 Hamburg, Germany}
\EmailD{\href{mailto:vicente.cortes@uni-hamburg.de}{vicente.cortes@uni-hamburg.de}}

\Address{$^{\rm b)}$~School of Computer \& Mathematical Sciences, University of Adelaide, SA 5005, Australia}
\EmailD{\href{mailto:thomas.leistner@adelaide.edu.au}{thomas.leistner@adelaide.edu.au}}

\ArticleDates{Received September 21, 2023, in final form September 09, 2024; Published online September 21, 2024}

\Abstract{A conformal transformation of a semi-Riemannian manifold is essential if there is no conformally equivalent metric for which it is an isometry. For Riemannian manifolds the existence of an essential conformal transformation forces the manifold to be conformally flat. This is false for pseudo-Riemannian manifolds, however compact examples of conformally curved manifolds with essential conformal transformation are scarce. Here we give examples of compact conformal manifolds in signature $(4n+2k,4n+2\ell)$ with essential conformal transformations that are locally conformally pseudo-K\"{a}hler and not conformally flat, where $n\ge 1$, $k, \ell \ge 0$. The corresponding local pseudo-K\"ahler metrics obtained by a local conformal rescaling are Ricci-flat.}


\Classification{53C50; 53C35; 53C18; 53C29}

\begin{flushright}
\begin{minipage}{60mm}
\em Dedicated to Christian B\"{a}r\\
on the occasion of his 60th birthday
\end{minipage}
\end{flushright}

\renewcommand{\thefootnote}{\arabic{footnote}}
\setcounter{footnote}{0}

\section{Introduction}

A {\em conformal diffeomorphism} between semi-Riemannian manifolds $(M,g)$ and $(N,h)$ is a diffeomorphism
$\phi\colon M\to N$ such that
\[\phi^*h=\e^{2\vf}g,\]
for a smooth function $\vf\in C^\infty(M)$. A {\em conformal transformation} is a conformal diffeomorphism from $(M,g)$ to itself.
The conformal transformations form a group and we call a group~$G$ of conformal transformations {\em essential} if there
if there is no metric $\hat g$ in the conformal class of $g$ such that~$G$ is a group of isometries for $\hat g$. Otherwise~$G$ is called {\em inessential}.
A single conformal transformation $\phi$ is called {\em essential} if the group generated by $\phi$ is essential, and inessential otherwise.
A {\em homothetic transformation} or just a {\em homothety} is a conformal transformation for which the function $\vf$ is constant $\vf\equiv\lambda\in \R$.
If $\lambda\not=0$, we say that $\phi$ is a {\em proper homothety}.

The Lichnerowicz conjecture states that if $(M,g)$ is a Riemannian manifold with an essential conformal transformation, then $(M,g)$ is either conformally diffeomorphic to $\R^n$ with the flat metric or to the sphere $\SS^n$ with the round metric. It was proved in a series of papers by Lelong-Ferrand and Obata (see \cite{Lelong-Ferrand71} and \cite{Obata71} for compact manifolds, \cite{Ferrand96} for non-compact ones) with contributions by \cite{Alekseevskii72}. It is clear that Euclidean space and the round sphere have essential conformal transformations: the conformal transformations of $\R^n$ are homotheties, whereas the conformal transformations of $\SS^n$ are induced by the linear isometries of Minkowski space $\R^{1,n+1}$. In fact, identifying the sphere $\SS^n$ with the projectivisation $\bigl\{ [x] \in \mathbb{R}P^{n+1} \mid x \text{ is null in }\R^{1,n+1}\bigr\}$ of the null cone, the conformal group of $\SS^n$ is precisely
$\mathrm{PO}(1,n+1)=\O(1,n+1)/\{\pm 1\}$. In both cases there are conformal transformations with a fixed point, which are essential (see~\mbox{\cite[Proposition~2.5]{LeistnerTeisseire22}} for a proof that non-trivial homotheties with a fixed point are essential -- we will use this fact below).
 It is remarkable that in the non-compact case, geodesic completeness does not have to be assumed for the statement to hold. This can be illustrated with Euclidean space $\R^n$. Here the homotheties $x\mapsto\e^s x$, for $s\not=0$, have a fixed point and are therefore essential. Removing the origin allows to define the metric \smash{$\tfrac{1}{\|x\|^2}g_{\mathrm{Euclid}}$} on $\R^n\setminus\{0\}$ for which these homotheties are isometries. By the geodesic incompleteness of $\R^n\setminus\{0\}$ one gains metrics that cause the loss of essential transformations.

 It is a natural question to which extent the result of Ferrand and Obata generalises to semi-Riemannian manifolds with {\em indefinite} metrics. It was clear from early on that it does not generalise to the non-compact context as there are plenty of non-compact Lorentzian manifolds with essential conformal transformations that are not conformally flat \cite{Alekseevski85,Podoksenov89,Podoksenov92}, and this holds beyond Lorentzian signature. It is much more difficult to find compact examples, but Frances in \cite{frances05} constructed compact Lorentzian manifolds with essential conformal transformations that are not conformally diffeomorphic to the homogeneous model of Lorentzian conformal geometry, the Einstein universe (see \cite{Frances08} for a survey). Recall that
 the Einstein universe of dimension $n$ is the coset space $\mathrm{PO}(2,n)/\mathrm{PO}(2,n)_L$ endowed with its canonical conformal structure, where $\mathrm{PO}(2,n)_L$ denotes the
 stabilizer of a null line $L\subset \R^{1,n+1}$ considered as a point $L\in \R P^{n+1}$. The group of conformal transformations
 of this manifold is $\mathrm{PO}(2,n)$ and its Weyl tensor vanishes.

These Lorentzian examples generalise to other signatures. Even though the examples found by Frances are not conformally diffeomorphic to the homogeneous model, they are conformally flat, i.e.,~have vanishing Weyl tensor. This led to a new {\em generalised Lichnerowicz conjecture}: a~compact semi-Riemannian manifold with essential conformal transformations is conformally flat. In \cite{frances12}, again Frances constructed counterexamples to this conjecture in all signatures {\em except Lorentzian}. In Lorentzian signature the generalised Lichnerowicz conjecture is known as {\em Lorentzian Lichnerowicz conjecture} and is still open, although substantial progress has been made recently in a series of papers by Frances, Melnick and Pecastaing \cite{FrancesMelnick21,MelnickPecastaing21,Pecastaing23,Pecastaing18}.

In the present article, we focus on signatures beyond Lorentzian and in particular on neutral signature. Our main result is the construction of Ricci-flat pseudo-K\"{a}hler symmetric spaces in signature $(4n,4n)$ that admit essential holomorphic homotheties and, based on this, we construct examples of compact locally conformally pseudo-K\"ahler manifolds $M_{a,b}$ with an essential conformal group that are locally conformally Ricci-flat, but not conformally flat. The latter result is stated in detail in Theorem~\ref{main:thm}. In addition, we consider real versions of these constructions in signature $(2n,2n)$ (see, e.g.,~Theorem~\ref{mainthm2}), and more generally $(2n+k,2n+\ell)$ that yield compact {conformal} manifolds with essential conformal transformations. The latter are generalisations of the examples constructed by Frances in \cite{frances12} in signatures $(2+k,2+\ell)$.

\section[A Hermitian symmetric space in signature (4n,4n) with compact conformal quotients]{A Hermitian symmetric space in signature $\boldsymbol{(4n,4n)}$ \\ with compact conformal quotients}
 In the following, we use the index conventions
\[A,B,C,\ldots \in \{1,\ldots, 4n\},\qquad a,b,c,\ldots\in \{1\ldots ,2n\},\qquad i,j,k,\ell,\ldots\in \{1,\ldots, n\}.\] We use the Einstein summation convention for these ranges of indices.

Let $(\sigma_{ij})$ be a symmetric complex $n\times n$ matrix and define the quadratic polynomial
\[\sigma\bigl(z^1,\ldots, z^{2n}\bigr):= \sigma_{ij}z^{i}z^{j+n},\]
on $\C^{2n}$.
Using $\sigma$, we define
 $f \colon \C^{4n}\to \rr$ as
\begin{align}
f\bigl(z^1,\ldots , z^{4n}\bigr)
&{}=
\delta_{ab} z^a\ol{z}^{b+2n}+
\sigma\bigl(z^1,\ldots, z^{2n}\bigr)\overline{\sigma\bigl(z^1,\ldots, z^{2n}\bigr)}
\nonumber\\
&{}=
\delta_{ab} z^a\ol{z}^{b+2n}
+
\sigma_{ij}\bar{\sigma}_{k\ell}\bigl(
 z^{i}z^{j+n}\ol{z}^{k} \ol{z}^{\ell+n}\bigr).\label{potential}
\end{align}
As K\"{a}hler potential, this function defines an indefinite K\"ahler metric of signature $(4n,4n)$ on $\C^{4n}=\rr^{8n}$ by
\begin{gather}\label{metric}
g= h_{AB} \,
 \d z^A \cdot \d \ol{z}^B,\qquad \text{where} \quad h_{AB}=\frac{\partial^2 f}{\partial z^A\partial \ol{z}^B},
 \end{gather}
with K\"{a}hler form
 \begin{align*}
\omega
=
\frac{\mathrm{i}}{2} h_{AB}\,
 \d z^A \wedge \d \ol{z}^B.
\end{align*}
The only non-vanishing terms in $h_{AB}$ are the following:
\begin{gather}
h_{a\, b+2n} = h_{a+2n\, b} = \delta_{ab},
\nonumber\\
h_{j\ell} =
\sigma_{ij}\bar{\sigma}_{k\ell}\,
 z^{i+n}\ol{z}^{k+n},
\nonumber\\
h_{j\, \ell+n} =
\sigma_{ij}\bar{\sigma}_{k\ell}\,
 z^{i+n}\ol{z}^{k},
\nonumber\\
h_{j+n\,\ell} =
\sigma_{ij}\bar{\sigma}_{k\ell}\,
 z^{i}\ol{z}^{k+n},
\nonumber\\
h_{j+n\, \ell+n} =
\sigma_{ij}\bar{\sigma}_{k\ell}\,
 z^{i}\ol{z}^{k},
\label{metriccoeff}
\end{gather}
so that {$(h_{AB})$} and its inverse are of the form
\[
{\bigl(h_{AB}\bigr)}=\begin{pmatrix} h_{ab} &\1_{2n} \\ \1_{2n} & 0\end{pmatrix},
\qquad
{\bigl(h^{AB}\bigr)}=\begin{pmatrix} 0 &\1_{2n} \\ \1_{2n} & -h_{ab}\end{pmatrix}.
\]

Without loss of generality, we can assume that the symmetric matrix {$(\sigma_{ij})$} is diagonalised, as the following lemma shows.
\begin{Lemma}\label{diag:lem}
For any complex symmetric $n\times n$ matrix $\sigma=(\sigma_{ij})$, the pseudo-Riemannian manifold $\bigl(\R^{8n},g\bigr)$ is isometric to a pseudo-Riemannian manifold $\bigl(\R^{8n},\hat g\bigr)$ of the same type defined by a diagonal matrix $(\hat\sigma_{ij})$.
\end{Lemma}

\begin{proof}
Assume that $Q\in \GL(n,\C)$ such that $\hat \sigma=Q^\top \sigma Q$ is diagonal. Let $P=\bigl(\overline{Q}^{-1}\bigr)^\top$. The real coordinate transformation $\phi_Q$ induced by complex linear transformation
\[ \hat{z}^{i}:=Q^i{}_k z^{k},\qquad \hat{z}^{i+n}:=Q^i{}_k z^{k+n},\]
and
\[ \hat{z}^{i+2n}:=P^i{}_k z^{k+2n},\qquad {\hat{z}^{i+3n}}:=P^i{}_k z^{k+3n},\]
maps $\R^{8n}$ to $\R^{8n}$ and pulls back the K\"{a}hler potential $f$ to the K\"{a}hler potential $\hat f= f\circ \phi_Q$,
\[
\hat f\bigl({z}^1,\ldots , {z}^{4n}\bigr)
=\delta_{ab} {z}^a\ol{{z}}^{b+2n}+
\hat\sigma_{ij}\bar{\hat\sigma}_{k\ell}
 \bigl(
 {z}^{i}{z}^{j+n}\ol{{z}}^{k} \ol{{z}}^{\ell+n}
 \bigr).
 \]
 As a consequence, the corresponding pseudo-K\"{a}hler metrics are isometric, $\hat g=\phi_Q^*g$. Note also that the quadratic term $\delta_{ab} {z}^a\ol{{z}}^{b+2n}$ is invariant under $\phi_Q$, i.e.,~$\phi_Q\in \O(4n,4n)$.
\end{proof}

We will now compute the Christoffel symbols for the metric (\ref{metric}) defined by the K\"{a}hler potential~(\ref{potential}). From (\ref{metriccoeff}), the metric is given as
\begin{gather}
g =\delta_{ab} \bigl(\d z^a\d \ol{z}^{b+2n} + \d z^{a+2n}\d \ol{z}^{b}\bigr)
\nonumber\\ \hphantom{g=}{}
 +
\sigma_{ij}\bar{\sigma}_{k\ell}
 \bigl(
 z^{i+n}\ol{z}^{k+n}\d z^{j}\d\ol{z}^{\ell}
 +
 z^{i+n}\ol{z}^{k}\d z^{j}\d\ol{z}^{\ell+n}
 +
 z^{i}\ol{z}^{k+n}\d z^{j+n}\d\ol{z}^{\ell}
 +
 z^{i}\ol{z}^{k}\d z^{j+n}\d\ol{z}^{\ell+n}\bigr)
\nonumber\\ \hphantom{g}{}
 =
\delta_{ab} \bigl(\d z^a\d \ol{z}^{b+2n} + \d z^{a+2n}\d \ol{z}^{b}\bigr)
+
 \sigma_{ij} \d\bigl(z^{i} z^{j+n}\bigr)
 \overline{
 \sigma_{k\ell} \d\bigl(z^{k} z^{\ell+n}\bigr)}
.\label{metriclong}
\end{gather}
In the following, we will write $\partial_A$ for $\frac{\partial}{\partial z^A}$ and $\partial_{\overline{A}}$ for $\frac{\partial}{\partial \overline{z}^A}$.
The {(essential)} Christoffel symbols of a K\"{a}hler metric are given by \cite[Section~IX.5]{ko-no2}
\[\Gamma^C_{AB} =h^{DC}\partial_Ah_{BD}.\]
\big(In fact, \smash{$\Gamma^{\bar{C}}_{\bar{A}\bar{B}}= \overline{\bigl(\Gamma^C_{AB}\bigr)}$} and the mixed Christoffel symbols are zero.\big)
In our situation, this shows~that
\[ \Gamma^A_{a+2n\ B}=0,\]
so that the holomorphic vector fields $\del_{a+2n}$, $a=1,\ldots , 2n$, are parallel on $\bigl(\R^{8n}, g\bigr)$. As~a~consequence, the real vector fields
$\del_{a+2n}+\del_{\overline{a+2n}}$ and $\mathrm{i}{\bigl(\del_{a+2n}-\del_{\overline{a+2n}}\bigr)}$ are parallel as well.

Similarly,
one
checks that
\[\Gamma^{c}_{AB} =0.\]
Furthermore, it is easy to see that
\[
\Gamma^{c+2n}_{i j}= \partial_{i}h_{j\, c}=0, \]
as well as
\[
\Gamma^{c+2n}_{i+n\,j+n}=0.
\]
A similar computation shows that the only non-vanishing Christoffel symbols are
\[
\Gamma^{\ell+2n}_{i+n\, j}
 =
 \partial_{i+n} h_{ j\, \ell}
 =
 \partial_{i+n} \bigl(
 \sigma_{mj}\bar{\sigma}_{k\ell}\,
 z^{m+n}\ol{z}^{k+n}\bigr)
 =
 \sigma_{ij}\bar{\sigma}_{k\ell}\,
\ol{z}^{k+n}
\]
and
 \[ \Gamma^{\ell+3n}_{i+n\, j}
 =
 \partial_{i+n} h_{ j\, \ell+n}
 =
 \partial_{i+n} \bigl(
 \sigma_{mj}\bar{\sigma}_{k\ell}\,
 z^{m+n}\ol{z}^{k} \bigr)
 =
 \sigma_{ij}\bar{\sigma}_{k\ell}\,
\ol{z}^{k} .\]
To summarise, for all $b\in \{1, \ldots, 2n\}$ and $A,B\in \{1, \ldots, 4n\}$, we have
\begin{equation}
\label{LC1}
\nabla_{\partial_A}\partial_{b+2n}=0,\qquad
\nabla_{\partial_A}\partial_b \in \Gamma (\mathrm{span} (\partial_{a+2n})_{a=1,\ldots, 2n}).\end{equation}
The curvature tensor is determined by ${\del_{a+2n}\hook R=0}$ and the following equations:
\begin{alignat*}{3}
&R\bigl(\del_{\overline{i}},\del_{j}\bigr)\del_{k} = 0,\qquad&& R\bigl(\del_{\overline{i}},\del_{j}\bigr)\del_{k+n} = \sigma_{jk}\sum_{\ell}\bar{\sigma}_{i\ell} \del_{\ell+3n},&
\\
&R\bigl(\del_{\overline{i+n}},\del_{j+n}\bigr)\del_{k} = \sigma_{jk}\sum_{\ell}\bar{\sigma}_{i\ell} \del_{\ell+2n},\qquad&& R\bigl(\del_{\overline{i+n}},\del_{j+n}\bigr)\del_{k+n} = 0,&
\\
&R\bigl(\del_{\overline{i}},\del_{j+n}\bigr)\del_{k} = \sigma_{jk}\sum_{\ell}\bar{\sigma}_{i\ell} \del_{\ell+3n},\qquad&& R\bigl(\del_{\overline{i}},\del_{j+n}\bigr)\del_{k+n} = 0,&
\\
&R\bigl(\del_{\overline{i+n}},\del_{j}\bigr)\del_{k} = 0,\qquad&& R\bigl(\del_{\overline{i+n}},\del_{j}\bigr)\del_{k+n} =
\sigma_{jk}\sum_{\ell}\bar{\sigma}_{i\ell} \del_{\ell+2n}.&
\end{alignat*}
Using conjugation and the skew symmetry of $R$, we get
\begin{alignat*}{3}
&R\bigl(\del_{\overline{i}},\del_{j}\bigr)\del_{\overline{k}} = 0,\qquad&& R\bigl(\del_{\overline{i}},\del_{j}\bigr)\del_{\overline{k+n}} = -\bar{\sigma}_{ik}\sum_{\ell}{\sigma}_{j\ell} \del_{\overline{\ell+3n}},&
\\
&R\bigl(\del_{\overline{i+n}},\del_{j+n}\bigr)\del_{\overline{k}} = -\bar{\sigma}_{ik}\sum_{\ell}{\sigma}_{j\ell} \del_{\overline{\ell+2n}}, \qquad&& R\bigl(\del_{\overline{i+n}},\del_{j+n}\bigr)\del_{\overline{k+n}} = 0,&
\\
&R\bigl(\del_{\overline{i}},\del_{j+n}\bigr)\del_{\overline{k}} = 0,\qquad&& R\bigl(\del_{\overline{i}},\del_{j+n}\bigr)\del_{\overline{k+n}} =
-\bar{\sigma}_{ik}\sum_{\ell}{\sigma}_{j\ell} \del_{\overline{\ell+2n}},&
\\
&R\bigl(\del_{\overline{i+n}},\del_{j}\bigr)\del_{\overline{k}} =
-\bar{\sigma}_{ik}\sum_{\ell}{\sigma}_{j\ell} \del_{\overline{\ell+3n}},\qquad&& R\bigl(\del_{\overline{i+n}},\del_{j}\bigr)\del_{\overline{k+n}} =0.&
\end{alignat*}
With this information at hand, we obtain the following proposition, for which we recall
that a~semi-Riemannian manifold is indecomposable if it does not split as a product, not even locally.

\begin{Proposition}\label{kaehlerprop}
The metric $g$ in \eqref{metric} defined by the K\"{a}hler potential \eqref{potential} is locally symmetric, K\"{a}hler of neutral signature, Ricci-flat, but not flat and hence not conformally flat. Moreover, the semi-Riemannian manifold $\bigl(\R^{8n},g\bigr)$ is a symmetric space.
\big(More precisely, $\bigl(\C^4=\R^{8n},g\bigr)$ is a Hermitian symmetric space.\big) If $(\sigma_{ij})$ is non-degenerate, then
$\bigl(\R^{8n},g\bigr)$ is indecomposable.
\end{Proposition}

\bprf
Using the fact~\cite[p.~90]{Moroianu07} that
\[ \mathrm{Ric}(\partial_A,\partial_{\bar B}) = -\partial_A\partial_{\bar B}\log \det (g(\partial_C,\partial_{\bar D}))),\]
we see that since the determinant of $h_{AB}$ is equal to $1$, $g$ is Ricci-flat.
With the above formulas~(\ref{LC1}) for the Levi-Civita connection, we have that
\begin{equation*}
\nabla_{\partial_A}\partial_B \in \Gamma (\mathrm{span} (\partial_{a+2n})_{a=1,\ldots, 2n}),\qquad \nabla_{\partial_{\bar{A}}}\partial_{\bar{B}} \in \Gamma \bigl(\mathrm{span} \bigl(\partial_{\overline{a+2n}}\bigr)_{a=1,\ldots, 2n}\bigr).\end{equation*}
Together with the fact that the components of the curvature tensor in the coordinates are constant, this implies that the curvature tensor is parallel, so that $g$ is locally symmetric,
$\nabla R=0$.

In order to show that $\R^{8n}$ with the metric $g$ is a globally symmetric space, we have to show that it is geodesically complete.
The geodesic equations are (we only work with the unbarred components, as the barred ones follow by complex conjugation)
\begin{gather*}
0 = \ddot{\gamma}^{a},\qquad\text{ i.e.,}\quad \gamma^a(t)=p^at+q^a,
\\
0 =
 \ddot{\gamma}^{\ell+2n} {+2}
 \sigma_{ij}\bar{\sigma}_{k\ell}\,
\ol{\gamma}^{k+n} \dot\gamma^{i+n}{\dot\gamma^{j}}
=
\ddot{\gamma}^{\ell+2n} {+2}
 \sigma_{ij}\bar{\sigma}_{k\ell}\,
\bigl(p^{k+n} t+q^{k+n}\bigr)
{p^{i+n}p^{j}},
\\
0 =
 \ddot{\gamma}^{\ell+3n} {+2}
 \sigma_{ij}\bar{\sigma}_{k\ell}\,
\ol{\gamma}^{k} \dot\gamma^{i+n}{\dot\gamma^{j}}
=
 \ddot{\gamma}^{\ell+3n} {+2}
 \sigma_{ij}\bar{\sigma}_{k\ell}\,
\bigl(p^{k} t+q^{k}\bigr){p^{i+n}p^{j}}.
\end{gather*}
Hence the $\gamma^{a+2n}(t) $ are cubic polynomials in $t$ and hence
 defined for all $t\in \R$.
As a consequence, $\bigl(\R^{8n},g\bigr)$ is geodesically complete and a symmetric space of signature $(4n,4n)$.

Now we prove that the holonomy algebra of $\bigl(\R^{8n},g\bigr)$ is indecomposable if $(\sigma_{ij})$ is non-degenerate.
Indecomposability means here that there is no decomposition of the (real) tangent space as a proper sum of two complementary non-degenerate invariant subspaces. This is equivalent to the geometric indecomposabilty formulated in the statement of the proposition.

In virtue of Lemma \ref{diag:lem}, we can assume that $(\sigma_{ij})$ is diagonal
with diagonal elements $\lambda_i\neq 0$.
The holonomy algebra is spanned by the following endomorphisms:
\begin{gather*}
R\bigl(\partial_i + \partial_{\bar i},\partial_j+\partial_{\bar j}\bigr) = R\bigl(\partial_{\bar i}, \partial_j\bigr) - R\bigl(\partial_{\bar j}, \partial_i\bigr),\\
R\bigl(\partial_i + \partial_{\bar i},\mathrm{i}\bigl(\partial_j-\partial_{\bar j}\bigr)\bigr) = \mathrm{i}\bigl(R\bigl(\partial_{\bar i}, \partial_j\bigr) +
R\bigl(\partial_{\bar j}, \partial_i\bigr)\bigr),\\
R\bigl(\partial_{i+n} + \partial_{\overline{i+n}},\partial_{j+n}+\partial_{\overline{j+n}}\bigr) = R\bigl(\partial_{\overline{i+n}}, \partial_{j+n}\bigr) - R\bigl(\partial_{\overline{j+n}}, \partial_{i+n}\bigr),\\
R\bigl(\partial_{i+n} + \partial_{\overline{i+n}},\mathrm{i}\bigl(\partial_{j+n}-\partial_{\overline{j+n}}\bigr)\bigr) = \mathrm{i}\bigl(R\bigl(\partial_{\overline{i+n}}, \partial_{j+n}\bigr) + R\bigl(\partial_{\overline{j+n}}, \partial_{i+n}\bigr)\bigr),\\
R\bigl(\partial_{i} + \partial_{\overline{i}},\partial_{j+n}+\partial_{\overline{j+n}}\bigr) = R\bigl(\partial_{\overline{i}}, \partial_{j+n}\bigr) - R\bigl(\partial_{\overline{j+n}}, \partial_{i}\bigr),\\
R\bigl(\partial_{i} + \partial_{\overline{i}},\mathrm{i}\bigl(\partial_{j+n}-\partial_{\overline{j+n}}\bigr)\bigr) = \mathrm{i}\bigl(R\bigl(\partial_{\overline{i}}, \partial_{j+n}\bigr) + R\bigl(\partial_{\overline{j+n}}, \partial_{i}\bigr)\bigr).
\end{gather*}
Inspection of the above formulas for these endomorphisms shows that the holonomy algebra $\mathfrak{hol}$ at the origin is represented
in the basis $(\partial_A)$ of the holomorphic tangent space ${U\!:=\!T_0^{1,0}\C^{4n}\cong \C^{4n}}$ by the algebra of $4n \times 4n$ matrices of the form
\[ \begin{pmatrix}0&0\\\ast &0\end{pmatrix},\]
where $\ast$ stands for an arbitrary skew-Hermitian $2n \times 2n$ matrix. We claim that the real vector space $U$ does not admit any
non-trivial decomposition into complementary subspaces invariant under $\mathfrak{hol}$. Let $U_1\subset U$ be a proper invariant subspace. It is either contained in
\[ U':= \mathrm{span}\{ \partial_{a+2n}\mid a=1,\ldots, 2n\},\]
which is precisely the kernel of $\mathfrak{hol}$, or has a nontrivial projection
to $U/U'$. Considering a line $L:=\R v$, where $v\in U_1 \setminus U'$, we see that in the latter case
\[ U'=\mathfrak{hol}\cdot L.\]
The reason is that the unitary group acts transitively on the unit sphere. By the holonomy invariance of $U_1$, this implies that $U'\subset U_1$, so that any
invariant subspace of $U$ is either contained in $U'$ or contains $U'$. This implies that the intersection
of any two invariant subspaces of complementary dimensions is non-trivial.
\eprf

\bs \label{transvec:prop} The group of transvections of the indecomposable Hermitian symmetric space $\bigl(\C^4,g\bigr)$ of Proposition~$\ref{kaehlerprop}$ is
{\rm 3}-step nilpotent with unipotent abelian isotropy. The isotropy algebra is of dimension $4n^2$.
\es

\bprf As a general fact, the Lie algebra $\mathfrak g$ of the transvection group of any symmetric space can be written as a direct sum of vector spaces,
\[ \mathfrak g= \mathfrak{hol} \+\mathfrak{m},\]
 where $\mathfrak m$ denotes the tangent space and $\mathfrak{hol}\subset \mathfrak{so}(\mathfrak m)$ the holonomy algebra of the symmetric space at the considered base point $o$ \big(in our case we take the origin in $\C^4$\big).
The Lie bracket is given by the conditions that $\mathfrak{hol}\subset \mathfrak g$ is a subalgebra (the Lie algebra of the isotropy group),
$\mathfrak{m}\subset \mathfrak{g}$ is a
$\mathfrak{hol}$-submodule with the adjoint action of $\mathfrak{hol}$ coinciding with the holonomy representation and, finally,
\[ [X,Y] = -R(X,Y)\in \mathfrak{hol}, \]
for all $X,Y\in \mathfrak{m}$. In the proof of Proposition~\ref{kaehlerprop}
we showed that $\mathfrak{hol}$ is represented by commuting lower triangular
matrices. Hence the isotropy algebra is abelian and unipotent. \big(Its dimension is clearly $\dim \mathfrak{hol} = 4n^2$, since
we saw that the vector space $\mathfrak{hol}$ is isomorphic to
the space of skew-Hermitian $2n \times 2n$ matrices.\big) We also see that
$[\mathfrak g, \mathfrak g] = \mathfrak{hol} + U'$, where $U'=\C^{2n}\subset \mathfrak{m}=\C^{4n}$ was defined in the proof of Proposition~\ref{kaehlerprop}.
Since $\mathfrak{hol}$ is abelian and annihilates $U'$, we have that $[\mathfrak g,[\mathfrak g, \mathfrak g]]= [\mathfrak{m}, \mathfrak{hol} + U']$. Using that $R(\mathfrak{m}, U') =0$, we arrive at \[ [\mathfrak g,[\mathfrak g, \mathfrak g]]= [\mathfrak{m}, \mathfrak{hol}] = U',\]
which proves that $\mathfrak g$ is 3-step nilpotent, since $[\mathfrak g, U'] =0$.
\eprf

As a consequence of $g$ being locally symmetric, the conformal group of the metric $g$ on $\rr^{8n}$ is contained in the group of homotheties,
\[H=\bigl\{ \phi\in \mathrm{Diff}\bigl(\R^{8n}\bigr)\mid \exists\ \lambda\in \rr\colon \psi^*g=\e^{2\lambda}g\bigr\}.\]
This is due to a result in \cite[Proposition 2.1]{CahenKerbrat77} that states that a conformal transformation between open sets in a semi-Riemannian manifold that is not conformally flat but has parallel Weyl tensor is a homothety. Since $g$ is locally symmetric it must also have parallel Weyl tensor, so that the result follows. Hence, we have the following.

\begin{Proposition}
The conformal group of $\bigl(\R^{8n},g\bigr)$, where $g$ is the metric in \eqref{metric} defined by the K\"{a}hler potential \eqref{potential} is equal to
\[\R\ltimes \Iso\bigl(\R^{8n},g\bigr),
\]
where the $\R$-factor is spanned by a one-parameter group of homotheties $\{\phi_s\mid s\in \R\}$.
Moreover, $g$~admits an abelian $2$-real-parameter family of homothetic transformations induced by the following diagonal linear maps of $\C^{4n}$,
\belabel{chomoth}\phi_{a,b}:=
\begin{pmatrix} \e^{a}\1_{n} &0&0&0\\ 0& \e^{b} \1_{n}&0&0 \\
0&0& \e^{a+2b} \1_{n}&0\\
0&0&0&\e^{2a+b}\1_{n}
\end{pmatrix},\qquad a,b\in \R,\end{equation}
that fix the origin $($as any linear map$)$ and satisfy $\phi_{a,b}^*g=\e^{2(a+b)}g$.

\end{Proposition}
\begin{proof}
We begin by showing that the metric $g$ admits a homothety that is not an isometry.
Indeed, for each $s\in \R$, the linear diagonal transformation
\[\phi_s:=\begin{pmatrix} \e^{s}\1_{2n} &0\\ 0& \e^{3s}\1_{2n}\end{pmatrix}\]
of $\C^{4n}$ induces a homothety of $g$ with $\phi_s^*g=\e^{4s}g$.
Composing the homothety $\phi_s$ with the isometry induced by
\[\psi_t:=\begin{pmatrix} \e^{t}\1_{n} &0&0&0\\ 0& \e^{-t} \1_{n}&0&0 \\
0&0& \e^{-t} \1_{n}&0\\
0&0&0&\e^{t}\1_{n}
\end{pmatrix},\]
we obtain the {two-parameter} family of homotheties in the proposition as
\[\phi_{a,b}= \phi_{\tfrac{a+b}{2}} \circ \psi_{\tfrac{a-b}{2}}. \]
This completes the proof.
\end{proof}

We will now construct a compact conformal manifold with essential conformal transformations.

\btheo\label{main:thm}
Let $\tM=\rr^{8n}\setminus\{0\}$ and $g$ the K\"ahler metric on $\tM$ defined by the potential $f$ in~\eqref{potential} for a given non-zero symmetric matrix $(\sigma_{ij})$.
Let $a,b>0$ be fixed positive real numbers and $\Gamma_{a,b} $ be the cyclic group of {holomorphic} homotheties of $(\tM,g)$ generated by $\phi_{a,b}$ in~\eqref{chomoth}. Then
$M_{a,b}=\tM/\Gamma_{a,b}$ is a compact manifold diffeomorphic to $\SS^1\times \SS^{8n-1}$ and carries a $($non-flat$)$ conformal structure of signature $(4n,4n)$ that is locally conformal to a Ricci-flat K\"ahler metric and has essential conformal transformations.
$($The induced integrable complex structure on the quotient manifold preserves the induced conformal structure.$)$
\etheo
\bprf

Since $a$ and $b$ are positive, $M_{a,b}=\tM/\Gamma_{a,b}$ is a compact manifold that is diffeomorphic to $\SS^1\times \SS^{8n-1}$.
In fact, a fundamental domain diffeomorphic to $(0,1)\times \SS^{8n-1}$ is given by
\[ \mathcal D = \Biggl\{ (x_1,x_2,x_3,x_4) \in \mathbb{R}^{8n} = \bigl(\mathbb{R}^{2n}\bigr)^4 \mid 1< \sum_{i=1}^4 \|x_i\|^2 \  \mbox{and}\
\sum_{i=1}^4 \frac{1}{\lambda_i^2} \|x_i\|^2<1\Biggr\},\]
 where $\lambda_1 = \e^a$, $\lambda_2 = \e^b$, $\lambda_3 = \e^{a+2b}$, $\lambda_4 = \e^{2a+b}$, and $\|\, \cdot\,\|$ is the Euclidean norm on $\R^{2n}$.
 Note that this is the region enclosed by the standard sphere and an ellipsoid. Identifying the two
 boundary components of $\mathcal D$ (i.e., the sphere and the ellipsoid) via the diffeomorphism induced by $\phi_{a,b}$ yields a diffeomorphism $M_{a,b} \cong \SS^1\times \SS^{8n-1}$.

{As} $\Gamma_{a,b}$ is a group of homotheties, the metric $g$ defines a conformal structure $\mathbf{c}$ on {$M_{a,b}$.} This conformal structure is locally conformally Ricci-flat K\"ahler and not conformally flat.

For $t\in \rr$, consider the homotheties
\[\phi_{0,t}=
\begin{pmatrix} \1_{n} &0&0&0\\ 0& \e^{t} \1_{n}&0&0 \\
0&0& \e^{2t} \1_{n}&0\\
0&0&0&\e^{t}\1_{n}
\end{pmatrix},
\]
of $\bigl(\tM,g\bigr)$ with fixed point set $\bigl\{\bigl(z^1,\ldots , z^n,0,\ldots, 0\bigr)\mid z^i \in \C^n\setminus \{0\}\bigr\}$. Since the $\phi_{0,t}$ have fixed points, they are essential on $\tM=\rr^{8n}\setminus\{0\}$, see \cite[Proposition 2.5]{LeistnerTeisseire22}.
In addition, for each~${t\in \rr}$, these homotheties commute with $\phi_{a,b}$ and hence descend to a conformal map $\psi_t$ on the quotient~$M_{a,b}$ and are essential on the quotient. Indeed, if there was a metric in the conformal class on~$M_{a,b}$ for which $\psi_t$ {was} an isometry, then this metric would lift to $\tM$ and have~$\phi_{0,t}$ as an isometry. This is a contradiction as the $\phi_{0,t}$ are essential on $\tM$.
\end{proof}

\bbem
The examples constructed in Theorem~\ref{main:thm} can be generalised to other dimensions and signatures $(4n+k,4n+\ell)$ by taking a metric on $\rr^{k+\ell+8n}$ as a product metric of $g$ and the Euclidean metric of signature $(k,\ell)$,
\smash{$g_{\mathrm{Euclid}}=-\sum_{i=1}^k \bigl(\d x^i\bigr)^2 + \sum_{j=k+1}^{k+\ell} \bigl(\d x^{j}\bigr)^2$},
 to obtain a~Ricci-flat metric. Moreover, when $k$ and $\ell$ are even, these examples are pseudo-K\"ahler.
The two-parameter family of homotheties on $\R^{8n}\times \R^{k,\ell}$ is then given by
\[ \phi_{a,b}\times \e^{a+b} \1_{k+\ell},\]
where $\phi_{a,b}$ is the homothety in~(\ref{chomoth}).
\ebem

\section[Metrics in signature (2n,2n) with essential conformal transformations]{Metrics in signature $\boldsymbol{(2n,2n)}$ \\ with essential conformal transformations}

In this section, we obtain compact conformal manifolds of signature $(2n,2n)$ with essential conformal transformations that are locally Ricci-flat and not conformally flat.
 For this, we replace in formula (\ref{metriclong}) for the K\"{a}hler metric the complex coordinates by real coordinates $x^i=z^i=\overline{z}^i$, for $i=1,\ldots, 4n$, on $\R^{4n}$, and we assume that the symmetric {$n\times n$ matrix $(\sigma_{ij})$ is real.} We use the same index conventions as in the previous section.
On $\R^{4n}$, we obtain the metric
\begin{equation}\label{metric4n}
g={2\delta_{ab}} \d x^a \d x^{b+2n}+\bigl( \sigma_{ij} \d\bigl(x^{i}x^{j+n}\bigr)\bigr)^2.
\end{equation}
Writing out the metric as
\[
g={2\delta_{ab}} \d x^a \d x^{b+2n}+
\sigma_{ij}\sigma_{k\ell} \bigl( x^{i}x^{k} \d x^{j+n} \d x^{\ell+n}
+ 2x^{i}x^{k+n} \d x^{j+n} \d x^{\ell}
+x^{i+n}x^{k+n} \d x^{j} \d x^{\ell}\bigr),
\]
one sees that the {$4n \times 4n$ matrix $(g_{AB})$} and its inverse are of the form
\[{(g_{AB})= \begin{pmatrix}
g_{ab} & \1_{2n}
\\
\1_{2n}
 &0\\
\end{pmatrix},\qquad \bigl(g^{AB}\bigr) =
\begin{pmatrix}
0& \1_{2n}
\\
\1_{2n}
 &-g_{ab} \\
\end{pmatrix},}\]
with {the coefficients $g_{ab}=g(\partial_a,\partial_b)$ determined by the above expansion of $g$.}
\begin{Remark}
It is important to note that the metric (\ref{metric4n}) is {\em not} the Hessian metric that is obtained from restricting the K\"{a}hler potential (\ref{potential}) to the real subspace $x^i=z^i=\overline{z}^i$ and by assuming $\sigma_{ij}$ to be real.
Indeed, the K\"{a}hler potential, when restricted to $z^i=\overline{z}^i$, becomes
\begin{align*}
f\bigl(x^1,\ldots , x^{4n}\bigr)
&{}=
\delta_{ab} x^a{x}^{b+2n}+
\bigl(\sigma\bigl(x^1,\ldots, x^{2n}\bigr)\bigr)^2
\nonumber\\
&{}=
\delta_{ab} x^a{x}^{b+2n}
+
\sigma_{ij}{\sigma}_{k\ell}
 \bigl(
 x^{i}x^k x^{j+n}{x}^{\ell+n}
 \bigr).
\end{align*}
The non-vanishing coefficients of the Hessian metric $h_{AB}=\partial_A\partial_B f$ defined by these coordinates and $f$ are
\begin{gather*}
h_{a\, b+2n} = h_{a+2n\, b} = \delta_{ab} { = g_{a\, b+2n} = g_{a+2n\, b}},
\\
h_{j\ell} =
2\sigma_{ij}{\sigma}_{k\ell}\,
 x^{i+n}{x}^{k+n} { = 2g_{j\ell}},
\\
h_{j\, \ell+n} =
2\left(\sigma_{j\ell}\sigma_{ik}+ \sigma_{ij}{\sigma}_{k\ell}\right)
 x^{i+n}{x}^{k} = h_{\ell+n j},
\\
h_{j+n\, \ell+n} = 2
\sigma_{ij}{\sigma}_{k\ell}\,
 x^{i}{x}^{k} { = 2g_{j+n\, \ell+n}.}
\end{gather*}
Note that the coefficients of the type $h_{j\, \ell+n}$ are (in general) not proportional to
$g_{j\, \ell+n}= \sigma_{ij}{\sigma}_{k\ell}
 x^{i+n}{x}^{k}$.
A computation then shows that the Hessian metric defined by {$(h_{AB})$} is in fact flat, whereas the metric $g$ in (\ref{metric4n}) is {(in general)} not, as we will see in the following.
\end{Remark}
Returning to the metric $g$ in (\ref{metric4n}), we note that the vector fields $\del_{2n+a}$, $a=1,\ldots , 2n$, are parallel and null, so for the Christoffel symbols we have
\[ \Gamma^A_{2n+a\,B}=0,\qquad \text{and}\qquad \Gamma^{c}_{AB} =0.\]
{Similar} computations as before show that
\[\Gamma^{A}_{i+n\ j+n}=\Gamma^{A}_{i j}=0,\]
and that the only non-vanishing Christoffel symbols are
 \[
 \Gamma^{k+2n}_{i\, j+n}
 =\frac{1}{2}( \partial_ig_{j+n\ k}+\partial_{j+n}g_{ik}-\partial_kg_{i\, j+n})
 =
 \sigma_{ij}\sigma_{k\ell}x^{\ell+n},
\]
and
\[
 \Gamma^{k+3n}_{i\, j+n}
 =
 \sigma_{ij}\sigma_{k\ell}x^{\ell}.
\]
To summarise, for all $b\in \{1, \ldots, 2n\}$ and $A,B\in \{1, \ldots, 4n\}$, we have
\begin{equation}
\label{LC1a}
\nabla_{\partial_A}\partial_{b+2n}=0,\qquad
\nabla_{\partial_A}\partial_b \in \Gamma (\mathrm{span} (\partial_{a+2n})_{a=1,\ldots, 2n}).\end{equation}
The curvature tensor is determined by the following equations:
\begin{alignat*}{3}
&\del_{2n+a}\hook R = 0,&& &
\\
&R(\del_{i+n},\del_{j+n})\del_{k} =
2\sum_{\ell}\sigma_{\ell [i}\sigma_{j]k}\del_{\ell+2n},
\qquad&& R(\del_{i+n},\del_{j+n})\del_{k+n} = 0,&
\\
&R(\del_{i},\del_{j})\del_{k} = 0,\qquad&& R(\del_{i},\del_{j})\del_{k+n} = 2\sum_{\ell}\sigma_{\ell [i}\sigma_{j]k}\del_{\ell+3n},&
\\
&R(\del_{i},\del_{j+n})\del_{k} = \sum_{\ell}\sigma_{\ell i}\sigma_{jk}\del_{\ell+3n},\qquad&& R(\del_{i},\del_{j+n})\del_{k+n} = -\sum_{\ell}\sigma_{\ell j}\sigma_{ik}\del_{\ell+2n}.&
\end{alignat*}
Here the square brackets denote skew-symmetrisation, $\sigma_{\ell [i}\sigma_{j]k}=\tfrac{1}{2} \left(\sigma_{\ell i}\sigma_{jk} - \sigma_{\ell j}\sigma_{ik}\right)$.
With this information at hand, we obtain the following proposition.
\begin{Proposition}\label{realex:prop}
Let $(\sigma_{ij})$ be any non-degenerate real symmetric $n\times n$ matrix. The metric~$g$ in~\eqref{metric4n} on $\rr^{4n}$ is locally symmetric, indecomposable, Ricci-flat, but not flat and hence not conformally flat. Moreover, the semi-Riemannian manifold $\bigl(\R^{4n},g\bigr)$ is an indecomposable symmetric space.
\end{Proposition}
\begin{proof}
With $R_{ABCD}$ constant and relations (\ref{LC1a}), the curvature tensor is parallel, so that $g$ is locally symmetric. The formulas for the curvature tensor also show that $g$ is Ricci-flat, but not flat. Hence, the metric is not conformally flat. In order to
analyse the holonomy algebra we can assume that $(\sigma_{ij})$ is diagonal. In fact, using a linear transformation
in $\GL(2n,\R) \subset \mathrm{O}(2n,2n)$, we can diagonalise $(\sigma_{ij})$, as we did in Lemma \ref{diag:lem} for the
Hermitian symmetric examples. Using the diagonal form, it is easy to check that the elements of $\hol(g)=\Span (R(\del_A,\del_B))_{{A,B=1,\ldots ,4n}}$ are of the form
\[\begin{pmatrix} 0&0 \\ A&0\end{pmatrix},\]
where $A$ is an arbitrary skew-symmetric $2n\times 2n$ matrix. This easily implies the indecomposability (by a simplified version of the argument used in the proof of Proposition~\ref{kaehlerprop}).

The proof that
 $\bigl(\R^{4n},g\bigr)$ is geodesically complete and hence a globally symmetric space is analogous to the proof for Proposition \ref{kaehlerprop}.
\end{proof}

Results analogous to the ones stated in Proposition~\ref{transvec:prop} hold for the indecomposable symmetric spaces
of Proposition~\ref{realex:prop}. In fact, the symmetric space is {again} determined by the general theory in terms of the algebra of (infinitesimal) transvections
\[ {\g=\mathfrak{hol} +\m},\]
where ${\mathfrak{hol}} =\Span \{ R(\del_a,\del_b)\}_{{a,b=1,\ldots, 2n}}$ is an abelian unipotent subalgebra of $\so(2n,2n)$
of dimension $n(2n-1)$ and $\m=\mathrm{span}(\del_{A})_{A=1,\ldots 4n}{=T_0\R^{4n}=\R^{4n}}$ with
the coordinate vector fields evaluated at the origin). As discussed above, the restriction of the Lie bracket to $\mathfrak m \times \mathfrak m \to \mathfrak{hol}$ is given by (minus) the curvature tensor. We see that $\a=\mathrm{span}(\del_{a+2n})_{a=1,\ldots 2n}\subset \m$ is the center of $\g$ and the derived algebra $\g' =[\mathfrak g , \mathfrak g] =\h+\a$ is abelian, so that $\g$ is 2-step solvable. Moreover, $[\mathfrak g , \mathfrak g'] = \a$ and $\mathfrak g$ is 3-step nilpotent. So we have proven the following
proposition.

\bs 
The group of transvections of the indecomposable symmetric space $\bigl(\R^{4n},g\bigr)$ of Proposition~$\ref{realex:prop}$ is {\rm 3}-step nilpotent with unipotent abelian isotropy. The isotropy algebra is of dimension $n(2n-1)$.
\es

As before, the linear maps in (\ref{chomoth}), this time acting on $\R^{4n}$, are homotheties of the metric~$g$ and define an essential conformal structure on the compact quotient of $\tM=\R^{4n}\setminus \{0\}$. Summarising, we obtain the
following theorem.
\btheo \label{mainthm2}Let $\widetilde{N}=\rr^{4n}\setminus\{0\}$ and $g$ the semi-Riemannian metric on $\widetilde{N}$ defined by \eqref{metric4n} for a~given non-zero real symmetric matrix $(\sigma_{ij})$.
Let $a,b>0$ be fixed positive real numbers and~$\Gamma_{a,b} $~be the cyclic group of homotheties of $(\widetilde{N},g)$ generated by $\phi_{a,b}$ in \eqref{chomoth}. Then
${N_{a,b}=\widetilde{N}/\Gamma_{a,b}}$ is a~compact manifold diffeomorphic to $\SS^1\times \SS^{4n-1}$ and carries a $($non-flat$)$ conformal structure of signature $(2n,2n)$ that is locally conformal to a Ricci-flat metric and has essential conformal transformations.
\etheo

Again, these examples can be generalised to signatures $(2n+k,2n+\ell)$ by taking the semi-Riemannian product with semi-Euclidean space $\R^{k,\ell}$.

For $n=1$, the symmetric spaces in Proposition~\ref{realex:prop} are isometric to the examples in \cite{frances12} of metrics in signature $(2,2)$ that yield compact quotients with essential conformal transformations. These examples are given by the metric
\begin{equation}\label{Francesmetric} \hat g= 2\d y^1 \d y^3 + 2\d y^2 \d y^4 + {\bigl(y^2\bigr)^2}\bigl(\d y^1\bigr)^2,\end{equation}
defining the structure of a symmetric space on $\R^4$. In fact, the geodesic completeness
can be proven as in the proof of Proposition~\ref{kaehlerprop} by computing the Christoffel symbols and writing out the geodesic equation
while $\nabla R=0$ is then obtained by computing also the curvature tensor.
The non-vanishing terms of the curvature tensor are
\[
R\biggl(\frac{\del}{\del y_{1}},\frac{\del}{\del y_{2}}\biggr)\frac{\del}{\del y_{1}}=\tfrac{\del}{\del y_{4}},\qquad R\biggl(\frac{\del}{\del y_{1}},\frac{\del}{\del y_{2}}\biggr)\frac{\del}{\del y_{2}}={-\tfrac{\del}{\del y_{3}}}.
\]
This is exactly the curvature tensor of the metric $g$ in (\ref{metric4n}) as computed above when
 $n=1$ and $\sigma_{ij}=1$. Hence with both spaces being symmetric and an isometry of the tangent spaces
 mapping the curvature tensors (at the origin) to each other, we conclude that the spaces are isometric.

\section{Conclusion and an open problem}
The main result of this article was the construction of counterexamples to the semi-Riemannian Lichnerowicz conjecture
that are locally conformally K\"ahler in all dimensions $\ge 8$. The index of an indefinite K\"ahler manifold is even and therefore
at least $2$. Our locally conformally K\"ahlerian counterexamples cover the cases of index $4$ and higher. It would be interesting to
decide if the Lichnerowicz conjecture holds for locally conformally K\"ahler manifolds of index $2$. That would be a natural analogue of the
Lorentzian Lichnerowicz conjecture in the presence of a complex structure.

\begin{Remark} The examples of symmetric spaces of signature $(2,2)$ obtained by specializing Proposition~\ref{realex:prop} to the case $n=1$ happen to be Hermitian symmetric but
the $2$-parametric group of homotheties (\ref{chomoth}) acting on $\R^{4}$, which was used to construct the compact quotients~$N_{a,b}$ of~$\R^4\setminus \{ 0\}$
and the essential conformal transformation,
does not preserve the complex structure. At the origin, the complex structure $J$ of the Hermitian symmetric space is given by
\[ J\frac{\partial}{\partial y^1} = \frac{\partial}{\partial y^2},\qquad J\frac{\partial}{\partial y^2} = -\frac{\partial}{\partial y^1},\qquad J\frac{\partial}{\partial y^3} = \frac{\partial}{\partial y^4},\qquad J\frac{\partial}{\partial y^4} = -\frac{\partial}{\partial y^3},\]
in the parametrization
\eqref{Francesmetric}. It is manifestly Hermitian with respect to $\hat{g}_0=2\d y^1 \d y^3 + 2\d y^2 \d y^4$ and invariant
under $\mathfrak{hol}$. This proves that the symmetric space is Hermitian symmetric.
\end{Remark}

\subsection*{Acknowledgements}
This work was supported by
the Australian Research
Council via the grant DP190102360 and
 by the German Science Foundation (DFG) under
the Research Training Group 1670 ``Mathematics inspired by String Theory'' and under
Germany's Excellence Strategy -- EXC 2121 ``Quantum Universe'' -- 390833306.
V.C.~is grateful to the University of Adelaide for its hospitality and support. V.C and T.L.
thank Australia's international and residential mathematical research institute MATRIX, where the work on the paper has started and was finalised.

\pdfbookmark[1]{References}{ref}
\LastPageEnding

\end{document}